\documentclass[12pt]{amsart}

\usepackage{enumerate, amsmath, amsthm, amsfonts, amssymb, xy,  mathrsfs, graphicx, paralist, fancyvrb}
\usepackage[usenames, dvipsnames]{xcolor}
\usepackage[margin=1in]{geometry} 
\usepackage[bookmarks, colorlinks=true, linkcolor=blue, citecolor=blue, urlcolor=blue]{hyperref}

\usepackage{eucal}

\input xy
\xyoption{all}

\numberwithin{equation}{section}
\newtheorem{theorem}[equation]{Theorem}

\newtheorem{proposition}[equation]{Proposition}
\newtheorem{lemma}[equation]{Lemma}
\newtheorem{corollary}[equation]{Corollary}

\theoremstyle{definition}
\newtheorem{rmk}[equation]{Remark}
\newenvironment{remark}[1][]{\begin{rmk}[#1] \pushQED{\qed}}{\popQED \end{rmk}}
\newtheorem{eg}[equation]{Example}

\newtheorem{defn}[equation]{Definition}
\newenvironment{definition}[1][]{\begin{defn}[#1]\pushQED{\qed}}{\popQED \end{defn}}

\newenvironment{subeqns}[1][]{\addtocounter{equation}{-1}
\begin{subequations}

}{\end{subequations}}

\newcommand{\cA}{\mathcal{A}}

\newcommand{\cB}{\mathcal{B}}

\newcommand{\rD}{\mathrm{D}}

\newcommand{\bF}{\mathbf{F}}

\newcommand{\rH}{\mathrm{H}}

\newcommand{\cI}{\mathcal{I}}

\newcommand{\cJ}{\mathcal{J}}

\newcommand{\cL}{\mathcal{L}}

\newcommand{\cO}{\mathcal{O}}

\newcommand{\bP}{\mathbf{P}}

\newcommand{\bS}{\mathbf{S}}
\newcommand{\cS}{\mathcal{S}}
\newcommand{\fS}{\mathfrak{S}}

\newcommand{\bV}{\mathbf{V}}

\newcommand{\bX}{\mathbf{X}}

\newcommand{\bZ}{\mathbf{Z}}

\newcommand{\bd}{\mathbf{d}}

\newcommand{\be}{\mathbf{e}}

\newcommand{\bk}{\mathbf{k}}

\newcommand{\bv}{\mathbf{v}}

\newcommand{\bw}{\mathbf{w}}

\newcommand{\bx}{\mathbf{x}}

\newcommand{\by}{\mathbf{y}}

\renewcommand{\phi}{\varphi}

\renewcommand{\tilde}[1]{\widetilde{#1}}

\newcommand{\arxiv}[1]{\href{http://arxiv.org/abs/#1}{{\tt arXiv:#1}}}

\DeclareMathOperator{\Sym}{Sym}
\DeclareMathOperator{\Spec}{Spec}

\newcommand{\GL}{\mathbf{GL}}
\newcommand{\SL}{\mathbf{SL}}

\newcommand{\init}{\mathrm{init}}
\newcommand{\Sec}{\mathrm{Sec}}
\newcommand{\Sub}{\mathrm{Sub}}

\title[Ideals of bounded rank symmetric tensors]{Ideals of bounded rank symmetric tensors \\are generated in bounded degree}
\author{Steven V Sam}
\thanks{SS was partially supported by NSF DMS-1500069}
\address{Department of Mathematics,
University of Wisconsin, Madison, WI, USA}
\email{\href{mailto:svs@math.wisc.edu}{svs@math.wisc.edu}}
\urladdr{\url{http://math.wisc.edu/~svs/}}
\date{April 19, 2016}


\begin{document}

\maketitle

\begin{abstract}
Over a field of characteristic zero, we prove that for each $r$, there exists a constant $C(r)$ so that the prime ideal of the $r$th secant variety of any Veronese embedding of any projective space is generated by polynomials of degree at most $C(r)$. The main idea is to consider the coordinate ring of all of the ambient spaces of the Veronese embeddings at once by endowing it with the structure of a Hopf ring, and to show that its ideals are finitely generated. We also prove a similar statement for partial flag varieties and, in fact, arbitrary projective schemes, and we also get multi-graded versions of these results.
\end{abstract}

\section{Introduction}

Given a vector space $V$ over a field of characteristic $0$, let $\rD^d V$ denote its $d$th divided power, i.e., the invariant subspace in $V^{\otimes d}$ under $\Sigma_d$ (symmetric group). The Veronese cone is the image of the map $V \to \rD^d V$ given by $v \mapsto v^{\otimes d}$. If $X \subseteq \bP(V)$ is a subvariety, the $r$th secant variety of the $d$th Veronese of $X$, denoted $\Sec_{d,r}(X)$, is the Zariski closure in $\rD^d V$ of the set of expressions $\sum_{i=1}^r v_i^{\otimes d}$ where $v_i$ is in the affine cone of $X$.

Secant varieties have long been a topic of interest in classical algebraic geometry. Despite this, very little is known about their algebraic structure. A fundamental invariant of an embedded variety is the maximum degree of a minimal generator of its defining ideal. Most known results about this invariant for secant varieties either restrict to small values of $r$ or are of an incremental nature which push the boundaries of existing techniques.

Heuristically, invariants of Veronese re-embeddings of projective varieties tend to exhibit strong uniformity properties. In line with this, the goal of this paper is to introduce new theoretical tools for the study of the ideals of their secant varieties, and as an application, to prove the following uniformity statement (this is a special case of Theorem~\ref{thm:secant-X-main}):

\begin{theorem} \label{thm:secant-X}
Let $X \subseteq \bP(V)$ be a subvariety. There is a function $C_X(r)$, depending on $r$ $($and $X)$, but independent of $d$, such that the ideal of $\Sec_{d,r}(X)$ is generated by polynomials of degree $\le C_X(r)$. 
\end{theorem}

An important special case is $X=\bP(V)$, i.e., secant varieties of Veronese embeddings of projective space. In this case, one can improve Theorem~\ref{thm:secant-X}:

\begin{theorem} \label{thm:secant-main}
There is a function $C(r)$, depending on $r$, but independent of $d$ and $\dim V$, such that the ideal of $\Sec_{d,r}(\bP(V))$ is generated by polynomials of degree $\le C(r)$.
\end{theorem}

The problem of finding generators for the ideal of $\Sec_{d,r}(\bP(V))$ is of general interest in the study of classical algebraic geometry and tensors since its points are the symmetric tensors with border rank $\le r$. This is closely related to the polynomial version of Waring's problem of expressing polynomials as sums of powers of linear forms. For this connection, see \cite{iarrobino}. 

Despite being the subject of intense research, very little is known about these equations: there are many partial results and only few special cases have complete answers. We refer the reader to \cite{landsberg-ottaviani} and its references for a survey of known results, and to \cite{landsberg} for an introduction to tensors in algebraic geometry. For example, the basic equations are the determinantal equations coming from catalecticant matrices, but these often do not cut out the secant variety, even set-theoretically \cite[Theorem 1.2]{BB} (see also \cite{BGL} for related discussions). These determinantal equations fit into the framework of ``vector bundle methods'' (see \cite[\S 5]{landsberg-ottaviani}), which essentially account for all known constructions of equations. These methods give an infinite hierarchy of types of equations, but still do not suffice in general \cite{galazka}.

\smallskip

The main idea of this paper is to combine all of the Veronese embeddings of $X$ into a single algebraic structure known as a Hopf ring (sometimes called coalgebraic ring to avoid confusion with ``Hopf algebra'') and to prove and exploit noetherianity properties of this structure. Hopf rings have two multiplications and a comultiplication satisfying some compatibilities. They are ubiquitous in algebraic topology, especially in the context of generalized homology of $\Omega$-spectra (see \cite{wilson} for an exposition). Hopf rings also naturally appear in the cohomology of families of finite groups \cite{GSS}. A combinatorial example is the ring of symmetric functions: the additional multiplication is inner plethysm; for its compatibility with multiplication and comultiplication, see \cite[\S I.7, Examples 22, 23]{macdonald}. As far as we are aware, Hopf rings have not appeared in the context of secant varieties.

Following Theorem~\ref{thm:secant-X}, it is natural to ask about the behavior of the function $r \mapsto C_X(r)$. Is it bounded by a polynomial? an exponential? We do not know, even in the case when $X = \bP(V)$. The advantage of our approach is that these questions can be rephrased in terms of the corresponding Hopf ring, which paves the way for a new line of investigation.

\subsection{Complements}

One generalization of Theorem~\ref{thm:secant-main} is to consider families of varieties that depend functorially on a vector space $V$. Our techniques allow us to handle the family of partial flag varieties. To be precise, let $\be = (e_1 < e_2 < \cdots < e_\ell)$ be an increasing sequence of positive integers. Given a vector space $V$, let $\bF(\be, V)$ be the partial flag variety of type $\be$: its points are increasing sequences of subspaces $W_1 \subset \cdots \subset W_\ell \subset V$ where $\dim W_i = e_i$. Its Picard group is isomorphic to $\bZ^\ell$, generated by  line bundles $\cL_1, \dots, \cL_\ell$ (corresponding to the fundamental weights $\omega_{e_1}, \dots, \omega_{e_\ell}$  of $\SL(V)$) with the property that $\cL(\bd) := \cL_1^{\otimes d_1} \otimes \cdots \otimes \cL_\ell^{\otimes d_\ell}$ is very ample if and only if $d_1, \dots, d_\ell > 0$ (and defines a map to projective space if and only if $d_1, \dots, d_\ell \ge 0$). When $\be = (1)$, $\bF(\be,V)$ is projective space.

Define $\Sec_{\bd,r}(\bF(\be,V))$ to be the $r$th secant variety of the image of $\bF(\be,V)$ under the map defined by the line bundle $\cL(\bd)$.

\begin{theorem} \label{thm:secant-flag}
There is a function $C_\be(r)$, depending on $r$ $($and $\be)$, but independent of $\bd$ and $\dim V$, such that the prime ideal of $\Sec_{\bd,r}(\bF(\be,V))$ is generated by polynomials of degree $\le C_\be(r)$.
\end{theorem}

\begin{remark}
The same proof applies in the more general setting of a projective scheme $X$ embedded in a product of projective spaces. There we work with the multi-homogeneous coordinate ring of $X$ and consider its multi-graded Veronese subrings which is the sum of components whose multi-degree is a multiple of a fixed multi-degree $\bd$. Then we get a boundedness statement as above which is independent of the choice of $\bd$.

The same techniques mentioned above allow us to get multi-graded generalizations of the main results, i.e., when considering products of projective varieties, see \S\ref{sec:products}.
\end{remark}

\begin{remark}
Going back to Theorem~\ref{thm:secant-X}, it is well-known that $C_X(r) \ge r+1$ \cite[Proposition 7.5.1.2]{landsberg} so one cannot remove the dependence on $r$. 
\end{remark}

\begin{remark}
The language of shuffle commutative algebras and Hopf rings used in this paper can be translated into the language of representations of categories, as in \cite{catgb}. The latter approach clarifies how to define projective generators for the category of modules (they are not just direct sums of the algebra) and will be investigated in later work in connection to higher syzygies of secant varieties. We have chosen to avoid this language here to keep the layers of abstraction to a minimum.
\end{remark}

\subsection{Outline of argument}

Even though Theorem~\ref{thm:secant-main} is a special case of the other results, we have chosen to present its proof carefully and then to explain the relevant changes for the other results because all of the key ideas are present in Theorem~\ref{thm:secant-main} and because we wanted to keep the notation as simple as possible. The proof of Theorem~\ref{thm:secant-main} breaks into the following steps:

\begin{enumerate}[\rm 1.]
\item For fixed $r$, we reduce to considering vector spaces of dimension $r$ (the main point is to reduce to a fixed dimension). For this, we can either use \cite{manivel-michalek} or the fact that the prime ideal of the associated ``subspace variety'' is generated in degrees that are independent of $d$ and $\dim V$ (in fact, it is generated by polynomials of degree $r+1$). This is explained in \S\ref{sec:proof}, and the remainder of the steps assume that we have made this reduction.

\item We consider all values of $d$ via the space $\cA_\Sigma = \bigoplus_{n,d} \Sym^n(\Sym^d V^*)$. The key observation is that this space has two products: the usual one which is ``external'' in that it multiplies the outside symmetric powers, and one which is ``internal'' in that it multiplies the inside symmetric powers. We show that when $\dim V < \infty$, subspaces of this space which are ideals for both products are finitely generated. The internal product involves symmetrizations and behaves poorly in positive characteristic, so here we assume that the field has characteristic $0$. This is done in \S\ref{sec:coalgebraic}; preparations for it are in \S\ref{sec:shuffle}.

\item Finally, the two products are compatible with the standard comultiplication on the symmetric algebra. Secant varieties are geometrically defined in terms of adding points, but we can also define them algebraically in terms of comultiplication. This allows us to prove the crucial fact that the ideals of the secant varieties of the Veronese cone are ideals in $\cA_\Sigma$ with respect to both products. So using the above, we can generate them using finitely many polynomials. This is explained in \S\ref{sec:joins}.
\end{enumerate}

The proof of Theorem~\ref{thm:secant-X} follows as in steps 2 and 3 above with $\Sym^d V^*$ replaced by $\rH^0(X; \cO(d))$ (and step 1 is not relevant), and is presented in \S\ref{sec:proof}; the proof of Theorem~\ref{thm:secant-flag} follows the same steps.

\subsection{Relation to previous work}

\begin{itemize}[$\bullet$]
\item Rather than look at all Veronese embeddings of projective space, one can consider all Segre embeddings of products of projective spaces, or one can consider all Pl\"ucker embeddings of Grassmannians. Are the ideals of their secant varieties defined in bounded degree? While this is unknown, if one only asks about generators of these ideals up to taking radical (i.e., finding equations which only cut out the secant variety as a set), then this is known by work of Draisma--Kuttler \cite{draismakuttler} and Draisma--Eggermont \cite{draismaeggermont}. Curiously, the techniques in those works do not apply to Veronese embeddings, even to get set-theoretic statements (as Draisma informed the author in private communication).

\item $\Sec_{d,1}(V)$ is the Veronese cone, and it is well known that its ideal is generated by quadratic polynomials, so $C(1) = 2$. Kanev shows in \cite{kanev} that $C(2) = 3$, that is, the ideal of the secant line variety of the Veronese cone is generated by cubic polynomials (see \cite{oeding-raicu, raicu} for related results).

\item Even ignoring the issue of considering Veronese re-embeddings, the problem of computing the ideals of secant varieties is a difficult one. We refer the reader to \cite{CGG, LW1, manivel-michalek, sidman} and the references therein. For some motivations for this problem coming from algebraic statistics, see \cite{GSS1}, and from geometric complexity theory, see \cite{landsberg-br, landsberg-matrix}.

\item The idea for showing that ideals in $\cA_\Sigma$ are finitely generated is motivated by the author's previous work on twisted commutative algebras (see \cite{symc1, expos, catgb, deg2tca}), which in turn was motivated by, for example, \cite{fimodules, cohen, DK, draismakuttler, draismaeggermont, higman, hillarsullivant, delta-mod}. These structures formalize the idea of working independently of the dimension of the underlying projective spaces, which is closely related to the notion of ``inheritance'' in \cite{landsberg-manivel}.

\item As we mentioned, properties of the coordinate ring of an embedded variety tend to improve as one takes sufficiently high Veronese subrings. For example, sufficiently high Veronese subrings are presented by quadrics and are Koszul (see \cite{backelin} and also \cite{ERT}). For an example with singularities of the secant variety and Veronese re-embeddings, see \cite{ullery}. 

\end{itemize}

\subsection{Conventions}

For the majority of the paper, $\bk$ denotes a field of characteristic $0$. In \S\ref{sec:shuffle}, we do not need this assumption, and we will let it be any commutative noetherian ring. All tensor products are taken over $\bk$.

The symmetric group on $n$ letters is $\Sigma_n$ and the set $\{1, \dots, n\}$ is denoted by $[n]$.

Given a vector space $V$, $\Sym^n V$ denotes its $n$th symmetric power, i.e., the coinvariants of $V^{\otimes n}$ under the usual permutation action of $\Sigma_n$, and $\Sym V = \bigoplus_{n \ge 0} \Sym^n V$. 

\subsection*{Acknowledgements}

I thank Jaros{\l}aw Buczy\'nski, Daniel Erman, Maciej Ga{\l}{\c a}zka, Mateusz Micha{\l}ek, Luke Oeding, Thanh Vu, and an anonymous referee for helpful comments.

\section{Shuffle commutative algebra} \label{sec:shuffle}

Fix a commutative noetherian ring $\bk$ (it will suffice to take $\bk$ to be a field, but the general case is no harder and may be useful for future applications).

\subsection{Shuffling polynomials} 
Fix a positive integer $r$. Set
\[
\cA = \bigoplus_{n, d \ge 0} (\Sym^d \bk^r)^{\otimes n}, \qquad \cA_{d,n} = (\Sym^d \bk^r)^{\otimes n}.
\]
We will only consider subspaces of $\cA$ which are homogeneous with respect to the bigrading $(d,n)$. A {\bf monomial} of $\cA$ is an element of the form $w_1 \otimes \cdots \otimes w_n$ where $w_i \in \Sym^d \bk^r$ and $w_i = \alpha x_{j_1} \cdots x_{j_d}$ where $\alpha \in \bk$ and $\{x_1, \dots, x_r\}$ is the standard basis of $\bk^r$. Pick a subset $\{i_1 < \cdots < i_n\}$ of $[n+m]$ and let $\{j_1 < \cdots < j_m\}$ be its complement; denote this pair of subsets by $\sigma$ and call it a {\bf split} of $[n+m]$. A split defines a shuffle product
\[
\cdot_\sigma \colon (\Sym^d \bk^r)^{\otimes n} \otimes (\Sym^d \bk^r)^{\otimes m} \to (\Sym^d \bk^r)^{\otimes (n+m)}
\]
where $(u_1 \otimes \cdots \otimes u_n) \cdot_\sigma (v_1 \otimes \cdots \otimes v_m) = w_1 \otimes \cdots \otimes w_{n+m}$ where $w_{i_k} = u_k$ and $w_{j_k} = v_k$. Whenever we write $f \cdot_\sigma g$, we are implicitly assuming that $\sigma$ is a split of the correct format (otherwise define it to be $0$). This gives $\cA$ the structure of a {\it shuffle commutative algebra}: we can think of the operations $\cdot_\sigma$ as ``shuffling'' the tensor factors together while preserving their relative orders. 

Commutativity means that we can always write $f \cdot_\sigma g = g \cdot_\tau f$ where $\tau$ is the split that swaps the two subsets in the split of $\sigma$. Associativity means that $f \cdot_\sigma (g \cdot_\tau h)$ can be rewritten as $(f \cdot_\alpha g) \cdot_\beta h$: taking the two products $\cdot_\sigma$ and $\cdot_\tau$ amounts to a nested splitting $S = S_1 \amalg (S_2 \amalg S_3)$, and to get $\cdot_\alpha$ and $\cdot_\beta$ we instead use the nested splitting $S = (S_1 \amalg S_2) \amalg S_3$.

In fact, $\cA$ has another multiplication: we use the obvious product
\[
* \colon (\Sym^d \bk^r)^{\otimes n} \otimes (\Sym^e \bk^r)^{\otimes n} \to (\Sym^{d+e} \bk^r)^{\otimes n}
\]
and extend it to all of $\cA$ by declaring that all other products are $0$. Then $*$ is both commutative and associative. The two products commute in a certain sense:

\begin{lemma} \label{lem:rewrite}
Given $f \in \cA_{d,n}$, $b \in \cA_{d,m}$, and $a \in \cA_{e,n+m}$ and a split  $\sigma$ of $[n+m]$, there exist $g_1,\dots,g_r \in \cA_{e,n}$ and $h_1,\dots,h_r \in \cA_{d+e, m}$ so that
\[
a * (b \cdot_\sigma f) = \sum_{i=1}^r h_i \cdot_\sigma (g_i * f).
\]
\end{lemma}

\begin{proof}
Since both $\cdot_\sigma$ and $*$ are bilinear, we may as well assume that $a,b$ are monomials; write $a = \alpha_1 \otimes \cdots \otimes \alpha_{n+m}$ and $b = \beta_1 \otimes \cdots \otimes \beta_m$. Suppose $\sigma$ is the split $\{i_1, \dots, i_m\}, \{j_1, \dots, j_n\}$ of $[n+m]$. Taking $g = \alpha_{j_1} \otimes \cdots \otimes \alpha_{j_n}$ and $h = \alpha_{i_1} \beta_1 \otimes \cdots \otimes \alpha_{i_m} \beta_m$ makes the identity valid.
\end{proof}

\begin{definition} \label{defn:shuffle-ideal}
A homogeneous subspace $I \subset \cA$ is an {\bf ideal} if $f \in I$ implies that $g * f \in I$ and $g \cdot_\sigma f \in I$ for all $g \in \cA$. A subset of elements of $\cA$ generates an ideal $I$ if $I$ is the smallest ideal that contains the subset.
\end{definition}

\begin{corollary} \label{cor:ideal-gen}
If $f_1, f_2, \dots$ generate an ideal $I$, then every element of $I$ can be written as a sum of elements of the form $h \cdot_\sigma (g * f_i)$ where $g,h \in\cA$.
\end{corollary}

An ideal is a {\bf monomial ideal} if it has a generating set consisting of monomials. The product of two monomials under either operation $*$ or $\cdot_\sigma$ is again a monomial; so equivalently, an ideal is a monomial ideal if it is linearly spanned by monomials. 

We put a partial ordering on monomials with coefficient $1$: $m \le m'$ if $m'$ is in the ideal generated by $m$. Using Corollary~\ref{cor:ideal-gen}, this is equivalent to saying that there exists monomials $n_0$ and $n_1$ such that $m' = n_1 \cdot_\sigma (n_0 * m)$. We can give another model for this partial ordering. Represent a monomial in $\Sym^d \bk^r$ by an element in $\bZ_{\ge 0}^r$ by taking its exponent vector. We partially order $\bZ_{\ge 0}^r$ pointwise, i.e., $(a_1, \dots, a_r) \le (b_1, \dots, b_r)$ if and only if $a_i \le b_i$ for $i=1,\dots,r$. Represent a monomial $m = w_1 \otimes \cdots \otimes w_n$ (with $w_i \in \Sym^d \bk^r$ a monomial) by a word $\bw(m)$ of length $n$ whose letters are elements in $\bZ_{\ge 0}^r$. Given two words $\bw = (w_1, \dots, w_n)$ and $\bw' = (w'_1, \dots, w'_m)$, we say that $\bw \le \bw'$ if there exist $1 \le i_1 < \cdots < i_n \le m$ such that $w_j \le w'_{i_j}$ for $j=1,\dots,n$. This is the same as asking for the existence of monomials $n_0, n_1$ as above. Call the resulting poset $(\bZ_{\ge 0}^r)^\star$.

\begin{proposition} \label{prop:Z-higman}
Given monomials $m, m'$, we have $m \le m'$ if and only if $\bw(m) \le \bw(m')$. Furthermore, $(\bZ_{\ge 0}^r)^\star$ is a noetherian poset, i.e., given any sequence $x_1, x_2, \dots$ of elements, there always exist $i < j$ such that $x_i \le x_j$.
\end{proposition}

\begin{proof}
We have already discussed the first part. The second part is an immediate consequence of Higman's lemma \cite[Theorem 1.3]{draisma-notes} since $\bZ_{\ge 0}^r$ is a noetherian poset (Dickson's lemma).
\end{proof}

\begin{corollary} \label{cor:monomial-fg}
Every monomial ideal of $\cA$ is finitely generated.
\end{corollary}

\begin{proof}
If not, then we can find a sequence of monomials $m_1, m_2, \dots$ such that $m_i$ is not in the ideal generated by $m_1, \dots, m_{i-1}$. Let $n_i$ be the monomial obtained from $m_i$ by replacing its coefficient with $1$. Then, by Proposition~\ref{prop:Z-higman}, there exists $i < j$ such that $n_j$ is in the ideal generated by $n_i$, i.e., $n_i \le n_j$. In fact, we can find infinitely many indices $i_1 < i_2 < \cdots$ such that $n_{i_j} \le n_{i_k}$ for all $j < k$ (this is a well-known property of noetherian posets, see \cite[Proposition 2.2]{catgb} for a proof). The ideal generated by the coefficients of $m_{i_1}, m_{i_2}, \dots$ is finitely generated since $\bk$ is noetherian, say by the coefficients of $m_{i_1}, \dots, m_{i_n}$. But then $m_{i_{n+1}}$ is in the ideal generated by $m_{i_1}, \dots, m_{i_n}$, which is a contradiction.
\end{proof}

Define a total ordering $\preceq$ on all monomials (ignoring their coefficients) of the same bidegree $(d,n)$ as follows. First, define $\preceq$ on $\bZ_{\ge 0}^r$ using lexicographic ordering, i.e., $(a_1, \dots, a_r) \preceq (b_1, \dots, b_r)$ if the first nonzero element of $(b_1 - a_1, \dots, b_r - a_r)$ is positive (note we are only comparing $a$ and $b$ if $\sum_i a_i = \sum_i b_i$). Then compare tensors using lexicographic ordering, i.e., $(w_1, \dots, w_n) \preceq (w_1', \dots, w'_n)$ if there exists $i$ such that $w_1 = w'_1, \dots, w_{i-1} = w'_{i-1}$ and $w_i \preceq w'_i$ but $w_i \ne w'_i$. We will only ever deal with bihomogeneous elements, so we do not need to worry about comparing elements of different bidegrees. This ordering is compatible with both products: 

\begin{lemma} \label{lem:order-prod}
Let $m, m', n$ be monomials. If $m \preceq m'$, then $n*m \preceq n*m'$ and $n \cdot_\sigma m \preceq n \cdot_\sigma m'$. 
\end{lemma}

Given $f \in \cA_{d,n}$, let $\init(f)$ be the largest monomial (together with its coefficient), with respect to $\preceq$, that has a nonzero coefficient in $f$. Given an ideal $I$, let $\init(I)$ be the $\bk$-span of $\{\init(f) \mid f \in I \text{ homogeneous}\}$.

\begin{lemma} \label{lem:monomial}
If $I$ is an ideal, then $\init(I)$ is a monomial ideal.
\end{lemma}

\begin{proof}
If $m \in \init(I)$, then write $m = \init(f)$ for $f \in I$. Given a monomial $n$, we have $n * m = \init(n * f)$ and $n \cdot_\sigma m = \init(n \cdot_\sigma f)$ by Lemma~\ref{lem:order-prod}, so $n*m \in \init(I)$ and $n \cdot_\sigma m \in \init(I)$. By bilinearity, the same is true for any $n$, so $\init(I)$ is an ideal.
\end{proof}

\begin{lemma} \label{lem:init-contain}
If $I \subseteq J$ are ideals and $\init(I) = \init(J)$, then $I=J$. In particular, if $f_1, f_2, \ldots \in J$ and $\init(f_1), \init(f_2), \dots$ generate $\init(J)$, then $f_1, f_2, \dots$ generate $J$.
\end{lemma}

\begin{proof}
Suppose $I$ is strictly contained in $J$. Pick $f \in J \setminus I$ with $\init(f)$ minimal with respect to $\preceq$. Then $\init(f) = \init(f')$ for some $f' \in I$. But $\init(f - f')$ is strictly smaller than $\init(f)$ and $f - f' \in J \setminus I$, so we have a contradiction.

For the second statement, take $I$ to be the ideal generated by $f_1, f_2, \dots$.
\end{proof}

\begin{corollary} \label{cor:sa-noeth}
Every ideal of $\cA$ is finitely generated.
\end{corollary}

\begin{proof}
Combine Corollary~\ref{cor:monomial-fg}, Lemma~\ref{lem:monomial}, and Lemma~\ref{lem:init-contain}.
\end{proof}

\subsection{Shuffling any ring} \label{ss:shuffle-any}

Let $B = \bigoplus_{d \ge 0} B_d$ be a graded commutative ring with $B_0 = \bk$. We can generalize the definition of $\cA$:
\[
\cS(B) = \bigoplus_{n, d \ge 0} B_d^{\otimes n}.
\]
The definitions of the shuffle products $\cdot_\sigma$ and of the $*$-product that we made for $\cA$ carry over to $\cS(B)$ without change. Ideals of $\cS(B)$ are defined the same way as in Definition~\ref{defn:shuffle-ideal}. Note that $\cS(\bk[x_1, \dots, x_r]) = \cA$. 

Given a graded ring homomorphism $\phi \colon A \to B$, we get a map $\cS(\phi) \colon \cS(A) \to \cS(B)$ which commutes with all shuffle products and the $*$-product. Furthermore, $\cS(\phi)$ is injective, respectively surjective, if $\phi$ has the corresponding property.

\begin{corollary} \label{cor:B-fg-ideal}
If $B$ is generated by $B_1$ as a $\bk$-algebra and $B_1$ is a finitely generated $\bk$-module, then every ideal of $\cS(B)$ is finitely generated.
\end{corollary}

\begin{proof}
By assumption, there is a surjective map $\bk[x_1, \dots, x_r] \to B$ for some $r$. In particular, there is a surjective map $\cA \to \cS(B)$. By Corollary~\ref{cor:sa-noeth}, the ideals of $\cA$ are finitely generated, so the same is true for $\cS(B)$.
\end{proof}

\section{Coalgebraic algebra} \label{sec:coalgebraic}

Now we assume that $\bk$ is a field of characteristic $0$. Let $B$ be a graded $\bk$-algebra which is generated by $B_1$ and such that $\dim_\bk B_1 < \infty$. Set $\cB = \cS(B)$. Define 
\begin{align*}
\cB^\Sigma = \bigoplus_{n, d \ge 0} (B_d^{\otimes n})^{\Sigma_n}, \qquad
\cB_\Sigma = \bigoplus_{n, d \ge 0} (B_d^{\otimes n})_{\Sigma_n},
\end{align*}
where $\Sigma_n$ acts by permuting the tensor factors, and the superscript and subscript denote taking invariants and coinvariants, respectively. (In the introduction, $\cA_\Sigma$ is $\cB_\Sigma$ when $B = \bk[x_1, \dots, x_r]$.) $\cB^\Sigma$ is a subalgebra of $\cB$ with respect to the $*$-product, but is not closed under the shuffle products. To fix this, define $f \cdot g = \sum_\sigma f \cdot_\sigma g$; then $\cB^\Sigma$ is closed under $\cdot$. Furthermore, $\cdot$ is both commutative and associative. Both are bigraded by $(d,n)$, let $\cB^\Sigma_{d,n}$ and $(\cB_\Sigma)_{d,n}$ denote the bigraded pieces. Again, we will only consider subspaces of $\cB^\Sigma$ and $\cB_\Sigma$ that are homogeneous with respect to the bigrading $(d,n)$.

For each $d,n$, define a linear projection 
\begin{align*}
\pi \colon \cB_{d,n} &\to \cB^\Sigma_{d,n}\\
w_1 \otimes \cdots \otimes w_n &\mapsto \frac{1}{n!} \sum_{\sigma \in \Sigma_n} w_{\sigma(1)} \otimes \cdots \otimes w_{\sigma(n)}.
\end{align*}
If $f \in \cB^\Sigma_{d,n}$, then $\pi(f) = f$, so $\pi$ is surjective. Set $\pi' = n! \pi$. We also denote the direct sum of these maps by $\pi \colon \cB \to \cB^\Sigma$ and $\pi' \colon \cB \to \cB^\Sigma$. Also define 
\begin{align*}
\fS \colon (\cB_{\Sigma})_{d,n} &\to \cB^\Sigma_{d,n}\\
w_1 \cdots w_n &\mapsto \sum_{\sigma \in \Sigma_n} w_{\sigma(1)} \otimes \cdots \otimes w_{\sigma(n)}.
\end{align*}
Then $\fS$ is a linear isomorphism, since $\frac{1}{n!} \fS$ is the inverse of the composition $\cB^\Sigma_{d,n} \to \cB_{d,n} \to (\cB_\Sigma)_{d,n}$. Denote the direct sum of these maps by $\fS \colon \cB_\Sigma \to \cB^\Sigma$.

\subsection{Properties of $\cB^\Sigma$}

\begin{lemma} \label{lem:pi-prop}
\begin{enumerate}[\rm (a)]
\item If $f \in \cB^\Sigma_{d,n}$ and $g \in \cB_{e,n}$ are homogeneous, then $\pi(g * f) = \pi(g) * f$.
\item If $f \in \cB_{d,n}$ and $g \in \cB_{d,m}$, then $\binom{n+m}{n} \pi(f \cdot_\sigma g) = \pi(f) \cdot \pi(g)$ for any split $\sigma$ of $[n+m]$.
\end{enumerate}
\end{lemma}

\begin{proof}
(a) Both $\pi(g * f)$ and $\pi(g) * f$ are bilinear in $g$ and $f$, so we may assume that $g = g_1 \otimes \cdots \otimes g_n$ for $g_i \in B_e$ and that $f = \sum_{\sigma \in \Sigma_n} f_{\sigma(1)} \otimes \cdots \otimes f_{\sigma(n)}$ for $f_i \in B_d$. Then 
\begin{align*}
\pi(g) * f &= \frac{1}{n!} \sum_{\sigma, \tau \in \Sigma_n} g_{\tau(1)} f_{\sigma(1)} \otimes \cdots \otimes g_{\tau(n)} f_{\sigma(n)},\\
\pi(g * f) &= \frac{1}{n!} \sum_{\sigma, \tau \in \Sigma_n} g_{\tau(1)} f_{\sigma \tau(1)} \otimes \cdots \otimes g_{\tau(n)} f_{\sigma \tau(n)}.
\end{align*}
But these sums are the same: in the second, do the change of variables $\sigma \mapsto \sigma \tau^{-1}$. We conclude that $\pi(g * f) = \pi(g) * f$.

(b) Again, both $\pi(f \cdot_\sigma g)$ and $\pi(f) \cdot \pi(g)$ are bilinear in $f,g$, so we may assume that $f = f_1 \otimes \cdots \otimes f_n$ with $f_i \in B_d$ and that $g = g_1 \otimes \cdots \otimes g_m$ with $g_i \in B_d$. To compute $\pi'(f \cdot_\sigma g)$, we first shuffle together $f$ and $g$ and symmetrize; note that this is the same as symmetrizing $f$ and $g$ individually and then shuffling them together in all possible ways. So
\[
\pi'(f \cdot_\sigma g) = \sum_\tau \pi'(f) \cdot_\tau \pi'(g) = \pi'(f) \cdot \pi'(g).
\]
The equation using $\pi$ in place of $\pi'$ follows immediately.
\end{proof}

\begin{definition}
A homogeneous subspace $I \subseteq \cB^\Sigma$ is an {\bf ideal} if $f \in I$ implies that $g \cdot f \in I$ and $g * f \in I$ for all $g \in \cB^\Sigma$.
\end{definition}

\begin{proposition} \label{prop:invt-noeth}
Every ideal of $\cB^\Sigma$ is finitely generated.
\end{proposition}

\begin{proof}
Given an ideal $J$ of $\cB^\Sigma$, let $I$ be the ideal in $\cB$ generated by $J$. By Corollary~\ref{cor:B-fg-ideal}, $I$ is finitely generated, say by $f_1, \dots, f_N$. We may assume that the $f_i$ belong to $J$. We claim that they also generate $J$ as an ideal in $\cB^\Sigma$. By Corollary~\ref{cor:ideal-gen}, every $f \in J$ can be written as a sum of terms of the form $h \cdot_\sigma (g * f_i)$ where $h, g \in \cB$. By Lemma~\ref{lem:pi-prop} we get $\pi(h \cdot_\sigma (g * f_i)) = \binom{n+m}{n}^{-1} \pi(h) \cdot (\pi(g) * f_i)$ where $h \in \cB_{d,n}$ and $g*f_i \in \cB_{d,m}$. Since $\pi(f) = f$, we conclude that every element of $J$ can be written as a sum of terms of the form $h' \cdot (g' * f_i)$ where $h', g' \in \cB^\Sigma$.
\end{proof}

For fixed $d$, $\bigoplus_n \cB_{d,n}^\Sigma$ is a free divided power algebra under $\cdot$, and hence is freely generated in degree $n=1$. So we can define a comultiplication $\Delta \colon \cB^\Sigma \to \cB^\Sigma \otimes \cB^\Sigma$ by $w \mapsto 1 \otimes w + w \otimes 1$ when $w \in \cB^\Sigma_{d,1}$, and requiring that it is an algebra homomorphism for $\cdot$.

For the following statement, we extend the products $\cdot$ and $*$ to $\cB^\Sigma \otimes \cB^\Sigma$ componentwise.

\begin{lemma} \label{lem:hopf-ring}
Pick $\bx \in \cB^\Sigma_{e,n}$, $\by \in \cB^\Sigma_{d,m}$, and $\bv \in \cB^\Sigma_{d,n}$. Then 
\begin{align*}
\Delta(\by \cdot \bv) &= \Delta(\by) \cdot \Delta(\bv)\\
\Delta(\bx * \bv) &= \Delta(\bx) * \Delta(\bv).
\end{align*}
\end{lemma}

\begin{subeqns}
\begin{proof}
The first identity follows from the way we defined $\Delta$.

For the second identity, first note that both $\Delta(\bx * \bv)$ and $\Delta(\bx) * \Delta(\bv)$ are bilinear in $\bx$ and $\bv$. So we can assume, without loss of generality, that $\bx = \sum_{\sigma \in \Sigma_n} x_{\sigma(1)} \otimes \cdots \otimes x_{\sigma(n)}$ for some $x_1 \otimes \cdots \otimes x_n \in \cB_{e,n}$ and that $\bv = \sum_{\sigma \in \Sigma_n} v_{\sigma(1)} \otimes \cdots \otimes v_{\sigma(n)}$ for some $v_1 \otimes \cdots \otimes v_n \in \cB_{d,n}$.

Then $\bv = v_1 \cdots v_n$ and $\bx = x_1 \cdots x_n$ (here we are using the $\cdot$ product). So 
\begin{align*}
\Delta(\bv) = \Delta(v_1) \cdots \Delta(v_n) &= \sum_{S \subseteq [n]} \pi'(v_S) \otimes \pi'(v_{[n] \setminus S})\\
\Delta(\bx) = \Delta(x_1) \cdots \Delta(x_n) &= \sum_{S \subseteq [n]} \pi'(x_S) \otimes \pi'(x_{[n] \setminus S})
\end{align*}
where the sum is over all subsets $S = \{s_1 < \cdots < s_j\}$ of $[n]$ and $v_S = v_{s_1} \otimes \cdots \otimes v_{s_j}$ and $x_S = x_{s_1} \otimes \cdots \otimes x_{s_j}$. This gives
\begin{align}
\Delta(\bx) * \Delta(\bv) &= (\sum_{T \subseteq [n]} \pi'(x_T) \otimes \pi'(x_{[n] \setminus T})) * (\sum_{S \subseteq [n]} \pi'(v_S) \otimes \pi'(v_{[n] \setminus S})) \notag \\
&= \sum_{S, T \subseteq [n]} \pi'(x_T * \pi'(v_S)) \otimes \pi'(x_{[n] \setminus T} * \pi'(v_{[n] \setminus S})), \label{eqn:Delta(b)Delta(x)}
\end{align}
where in the second equality we used Lemma~\ref{lem:pi-prop}(a). On the other hand,
\begin{align*}
\bx * \bv &= \sum_{\sigma, \tau \in \Sigma_n} x_{\sigma(1)} v_{\tau(1)} \otimes \cdots \otimes x_{\sigma(n)} v_{\tau(n)}\\
&= \sum_{\sigma \in \Sigma_n} x_{\sigma(1)} v_1 \cdots x_{\sigma(n)} v_n,
\end{align*}
where in the second sum we are using the $\cdot$ product. In particular, 
\begin{align} \label{eqn:Delta(bx)}
\Delta(\bx * \bv) &= \sum_{\sigma \in \Sigma_n} \sum_{S \subseteq [n]} \pi'(xv_{\sigma, S}) \otimes \pi'(xv_{\sigma, [n] \setminus S})
\end{align}
where $xv_{\sigma,S} = x_{\sigma(s_1)} v_{s_1} \otimes \cdots \otimes x_{\sigma(s_j)} v_{s_j}$ if $S = \{s_1 < \cdots < s_j\}$. We also write $[n] \setminus S = \{s_{j+1} < \cdots < s_n\}$.

We now show that the expressions \eqref{eqn:Delta(b)Delta(x)} and \eqref{eqn:Delta(bx)} for $\Delta(\bx) * \Delta(\bv)$ and $\Delta(\bx * \bv)$ can be identified. For \eqref{eqn:Delta(b)Delta(x)}, we are summing over expressions of the form
\[
\begin{array}{l|r|l}
x  &t_{\sigma(1)} \cdots t_{\sigma(j)} & t_{\alpha(j+1)} \cdots t_{\alpha(n)} \\
v & s_{\tau \sigma(1)} \cdots s_{\tau \sigma(j)} & s_{\beta \alpha(j+1)} \cdots s_{\beta \alpha(n)}
\end{array},
\]
which is shorthand for 
\[
(x_{t_{\sigma(1)}} v_{s_{\tau\sigma(1)}} \otimes \cdots \otimes x_{t_{\sigma(j)}} v_{s_{\tau \sigma(j)}}) \otimes 
(x_{t_{\alpha(j+1)}} v_{s_{\beta \alpha(j+1)}} \otimes \cdots \otimes x_{t_{\alpha(n)}} v_{s_{\beta \alpha(n)}}),
\]
and here $\sigma, \tau \in \Sigma_j$ and $\alpha, \beta \in \Sigma_{n-j}$ are freely chosen and $S$ and $T$ are freely chosen subsets of size $j$. We can do a change of variables $\rho = \tau \sigma$ and $\gamma = \beta \alpha$ to get 
\begin{align} \label{eqn:Deltabx-2}
\begin{array}{l|r|l}
x  &t_{\sigma(1)} \cdots t_{\sigma(j)} & t_{\alpha(j+1)} \cdots t_{\alpha(n)} \\
v & s_{\rho(1)} \cdots s_{\rho(j)} & s_{\gamma(j+1)} \cdots s_{\gamma(n)}
\end{array}.
\end{align}
On the other hand, \eqref{eqn:Delta(bx)} is a sum over expressions of the form
\begin{align} \label{eqn:Deltabx-3}
\begin{array}{l|r|l}
x & \lambda(s_{\rho(1)}) \cdots \lambda(s_{\rho(j)}) & \lambda(s_{\gamma(j+1)}) \cdots \lambda (s_{\gamma(n)}) \\
v & s_{\rho(1)} \cdots s_{\rho(j)} & s_{\gamma(j+1)} \cdots s_{\gamma(n)}
\end{array}
\end{align}
where $\rho \in \Sigma_j$, $\gamma \in \Sigma_{n-j}$, $\lambda \in \Sigma_n$ are chosen freely and $S$ is a freely chosen subset of size $j$. To identify this with \eqref{eqn:Deltabx-2}, set $\{t_1, \dots, t_j\} = \{\lambda(s_{\rho(1)}), \dots, \lambda(s_{\rho(j)})\}$ and $\{t_{j+1}, \dots, t_n\} = \{\lambda(s_{\gamma(j+1)}), \dots, \lambda(s_{\gamma(n)})\}$ (there is a unique identification so that $t_1 < \cdots < t_j$ and $t_{j+1} < \cdots < t_n$). Now there is a unique choice of $\sigma \in \Sigma_j$ and $\alpha \in \Sigma_{n-j}$ so that \eqref{eqn:Deltabx-2} and \eqref{eqn:Deltabx-3} agree, so we conclude that $\Delta(\bx * \bv) = \Delta(\bx) * \Delta(\bv)$.
\end{proof}
\end{subeqns}

We can now present a symmetrized version of Lemma~\ref{lem:rewrite}:
\begin{lemma}
Given $f \in \cB_{d,n}^\Sigma$, $b \in \cB_{d,m}^\Sigma$, and $a \in \cB_{e,n+m}^\Sigma$, we have
\[
a * (b \cdot f) = \Delta(a) * (b \otimes f).
\]
\end{lemma}

\begin{proof}
We follow the proof of Lemma~\ref{lem:rewrite}: assume that $a = \sum_{\rho \in \Sigma_{n+m}} \alpha_{\rho(1)} \otimes \cdots \otimes \alpha_{\rho(n+m)}$ and $b = \sum_{\tau \in \Sigma_m} \beta_{\tau(1)} \otimes \cdots \otimes \beta_{\tau(n)}$. Write $a^\rho = \alpha_{\rho(1)} \otimes \cdots \otimes \alpha_{\rho(n+m)}$ and $b^\tau = \beta_{\tau(1)} \otimes \cdots \otimes \beta_{\tau(n)}$. Pick a split $\sigma = \{i_1, \dots, i_m\}, \{j_1, \dots, j_n\}$ of $[n+m]$. Then 
\[
a^\rho * (b^\tau \cdot_\sigma f) = (\alpha_{\rho(i_1)} \otimes \cdots \otimes \alpha_{\rho(i_m)} * \beta_{\tau(1)} \otimes \cdots  \otimes \beta_{\tau(m)}) \cdot_\sigma (\alpha_{\rho(j_1)} \otimes \cdots \otimes \alpha_{\rho(j_n)} * f),
\]
and $a * (b \cdot_\sigma f)$ is the sum of these expressions over all choices of $\rho, \tau, \sigma$. But this is the same thing as $\Delta(a) * (b \otimes f)$.
\end{proof}

\begin{remark}
In the previous lemma, write $\Delta(a) = \sum a^{(1)} \otimes a^{(2)}$. Then 
\[
a * (b \cdot f) = \Delta(a) * (b \otimes f) = \sum (a^{(1)} * b) \cdot (a^{(2)} * f).
\]
So we conclude that $*$ ``distributes'' over $\cdot$ in the sense of coalgebras. More specifically, $\cB^\Sigma$ is a ring object in the monoidal category of coalgebras, where $*$ is multiplication and $\cdot$ is addition. These are called Hopf rings or coalgebraic rings (we have ignored the antipode since we do not use it). This is closely related to the example in \cite[Definition 2.1]{GSS}.
\end{remark}

\subsection{Properties of $\cB_\Sigma$}

$\cB_\Sigma$ has an algebra structure: define $f \cdot g$ to be the image of $\tilde{f} \cdot_\sigma \tilde{g}$ under $\cB \to \cB_\Sigma$ for any split $\sigma$ and lifts $\tilde{f}, \tilde{g} \in \cB$ of $f,g \in \cB_\Sigma$ (this is independent of the choice of lifts and the choice of split). To define a $*$-product on $\cB_\Sigma$, we use the fact that $\fS$ is a linear isomorphism and define $f * g = \fS^{-1}(\fS(f) * \fS(g))$.

\begin{remark}
While $\fS$ behaves poorly in positive characteristic, the $*$-product can still be defined on $\cB_\Sigma$. Consider the universal case $\bk = \bZ$, in which case $\fS$ is injective. One can check that the image is a subring under the $*$-product, so it can be uniquely pulled back to a product on $\cB_\Sigma$. Via base change, we get a product defined for any commutative ring $\bk$.

Unfortunately, many products in $\cB_\Sigma$ are $0$ when $\bk$ is a field of positive characteristic: for instance, when $B = \bk[x]$, then $[x^d]^n$ spans $\Sym^n(B_d)$, and $[x^d]^n * [x^e]^n = n! [x^{d+e}]^n$.
\end{remark}

Also, $\cB_\Sigma$ has a comultiplication $\Delta$ defined as follows: if $w_1 \cdots w_n \in \Sym^n(B_d)$, then 
\[
\Delta(w_1 \cdots w_n) = \sum_{S \subseteq [n]} w_S \otimes w_{[n]\setminus S}
\]
where the sum is over all subsets $S \subseteq [n]$ and $w_S = \prod_{i \in S} w_i \in \Sym^{|S|}(B_d)$. Most importantly, this is defined by $w \mapsto w \otimes 1 + 1 \otimes w$ when $w \in \Sym^1(B_d)$ and extended uniquely by the requirement that $\Delta \colon \cB_\Sigma \to \cB_\Sigma \otimes \cB_\Sigma$ is an algebra homomorphism. 

\begin{proposition}
The symmetrization map $\fS \colon\cB_\Sigma \to \cB^\Sigma$ is an isomorphism of bigraded bialgebras under the $\cdot$ product. More precisely, the following two diagrams commute:
\[
\xymatrix{ \cB_\Sigma \otimes \cB_\Sigma \ar[r]^-\cdot \ar[d]_-{\fS \otimes \fS} & \cB_\Sigma \ar[d]^-\fS \\
\cB^\Sigma \otimes \cB^\Sigma \ar[r]^-\cdot & \cB^\Sigma} \qquad
\xymatrix{ \cB_\Sigma \ar[r]^-\Delta \ar[d]_-{\fS} & \cB_\Sigma \otimes \cB_\Sigma \ar[d]^-{\fS \otimes \fS} \\
\cB^\Sigma \ar[r]^-\Delta & \cB^\Sigma \otimes \cB^\Sigma}
\]
\end{proposition}

\begin{proof}
To verify that $\fS$ is compatible with $\cdot$, we just need to check on monomials. We have 
\begin{align*}
\fS(w_1 \cdots w_n) \cdot \fS(v_1 \cdots v_m) = (\sum_{\sigma \in \Sigma_n} w_{\sigma(1)} \otimes \cdots \otimes w_{\sigma(n)}) \cdot (\sum_{\tau \in \Sigma_m} v_{\tau(1)} \otimes \cdots \otimes v_{\tau(m)}).
\end{align*}
Symmetrizing the first $n$ terms and symmetrizing the last $m$ terms and then summing over all splits is the same as symmetrizing over all $n+m$ terms, so this sum is the same as $\fS(w_1 \cdots w_n v_1 \cdots v_n)$, which shows that $\fS$ is compatible with $\cdot$.

Both $\cB_\Sigma$ and $\cB^\Sigma$ are bialgebras generated in degree $1$ and $\fS$ is compatible with comultiplication in degree $1$, so the same is true for higher degrees. So $\fS$ is an isomorphism.
\end{proof}

In particular, if $\cI \subseteq \cB_\Sigma$ is an ideal under $\cdot$, then the same is true for $\fS(\cI) \subseteq \cB^\Sigma$. 

\begin{definition}
An ideal $\cI \subseteq \cB_\Sigma$ is a {\bf di-ideal} if $\fS(\cI)$ is an ideal with respect to $*$.
\end{definition}

We will explain some basic constructions for di-ideals shortly.

\section{Joins and secants} \label{sec:joins}

Let $V$ be a vector space and $\Sym(V)$ be its symmetric algebra. Given ideals $I, J \subset \Sym(V)$, their {\bf join} $I \star J$ is the kernel of
\[
\Sym(V) \xrightarrow{\Delta} \Sym(V) \otimes \Sym(V) \to \Sym(V)/I \otimes \Sym(V)/J,
\]
where the first map is the standard comultiplication (which is dual to the addition map on $\Spec(\Sym(V))$). Note that $\star$ is an associative and commutative operation since $\Delta$ is coassociative and cocommutative. Set $I^{\star 1} = I$ and $I^{\star r} = I \star I^{\star (r-1)}$ for $r > 1$.

\begin{proposition} \label{prop:join-reduced}
Assume $\bk$ is an algebraically closed field. If $I$ and $J$ are radical ideals, then $I \star J$ is a radical ideal. If $I$ and $J$ are prime ideals, then $I \star J$ is a prime ideal.
\end{proposition}

\begin{proof}
$\Sym(V) / (I \star J)$ is a subring of $\Sym(V) / I \otimes \Sym(V) / J$. The tensor product of reduced rings, respectively integral domains, is also reduced, respectively an integral domain, if the rings are finitely generated over an algebraically closed field \cite[Proposition 5.17]{milne}. Finally, the property of being reduced or integral is inherited by subrings.
\end{proof}

These definitions make sense for ideals $\cI, \cJ \subseteq \cB_\Sigma$ (continuing the notation for $\cB$ from the last section), so we can define the join $\cI \star \cJ$. To be precise, $(\cI \star \cJ)_{d,n}$ is the kernel of the map
\[
(\cB_\Sigma)_{d,n} \xrightarrow{\Delta} \bigoplus_{i=0}^n (\cB_\Sigma / \cI)_{d,i} \otimes (\cB_\Sigma / \cJ)_{d,n-i}.
\]

Since $\fS$ is compatible with $\Delta$, we deduce that 
\[
\fS(\cI \star \cJ) = \fS(\cI) \star \fS(\cJ).
\]

\begin{proposition} \label{prop:join-biideal}
If $\cI, \cJ \subseteq \cB_\Sigma$ are di-ideals, then $\cI \star \cJ$ is a di-ideal.
\end{proposition}

\begin{proof}
Pick $\bv \in \fS(\cI \star \cJ)$. By definition, $\bv$ is in the kernel of the map
\[
\Delta \colon \cB^\Sigma \to \cB^\Sigma / \fS(\cI) \otimes \cB^\Sigma / \fS(\cJ).
\]
Since both $\fS(\cI)$ and $\fS(\cJ)$ are ideals under $*$, it gives a well-defined multiplication on $\cB^\Sigma / \fS(\cI) \otimes \cB^\Sigma / \fS(\cJ)$. By Lemma~\ref{lem:hopf-ring}, given $\bx \in \cB^\Sigma$, we have $\Delta(\bx * \bv) = \Delta(\bx) * \Delta(\bv)$. But $\Delta(\bv) = 0$, so $\bx* \bv \in \fS(\cI \star \cJ)$. 
\end{proof}

Given a graded ring $B$ generated by $B_1$, and a positive integer $d$, we have an ideal in $\Sym(B_d)$, which is the kernel of the surjection $\Sym(B_d) \to \bigoplus_{e \ge 0} B_{de}$. Putting these all together, we get a subspace $\cI_B(1) \subset \cB_\Sigma$, that is, $\cI_B(1)_{d,n}$ is the kernel of the map $\Sym^n(B_d) \to B_{dn}$. We define $\cI_B(r) = \cI_B(1)^{\star r}$.

\begin{proposition} \label{prop:secant-biideal}
For any $r$, $\cI_B(r)$ is a di-ideal.
\end{proposition}

\begin{proof}
Using Proposition~\ref{prop:join-biideal}, it suffices to do the case $r=1$. Pick $f \in \cI_B(1)_{d,n}$. Then $\fS(f) \in B_d^{\otimes n}$ maps to $0$ under the multiplication map $\mu \colon B_d^{\otimes n} \to B_{dn}$. Now pick $g \in \cB^\Sigma_{e,n}$; we want to show that $\fS(f) * g \in \fS(\cI_B(1))_{d+e,n}$. This follows from the fact that the following diagram commutes
\[
\xymatrix{
B_d^{\otimes n} \otimes B_e^{\otimes n} \ar[r]^-{*} \ar[d]_-{\mu \otimes \mu} & B_{d+e}^{\otimes n} \ar[d]^-{\mu} \\
B_{dn} \otimes B_{en} \ar[r]^-\mu & B_{(d+e)n}
}
\]
since the composition of both pairs of arrows is the map that multiplies everything together using the product in $B$.
\end{proof}

\section{Proof of the main results} \label{sec:proof}

Let $V$ be a vector space. A subscheme $X \subseteq V$ is {\bf conical} if its defining ideal $I_X$ is homogeneous. There are two possible definitions of the $d$th Veronese embedding of $X$, or more specifically, of its coordinate ring. The first is to projectivize $X$ and take $\bigoplus_{e \ge 0} \rH^0(X; \cO_X(de))$ and the second is $\bigoplus_{e \ge 0} (\Sym(V^*)/I_X)_{de}$. These two definitions agree when the embedding of $X$ is projectively normal, i.e., the multiplication maps $\Sym^n \rH^0(X; \cO_X(1)) \to \rH^0(X; \cO_X(n))$ are surjective. In general, we will take the second definition. 

The {\bf $r$th secant scheme} of $X$ is the subscheme of $V$ defined by the ideal $I_X^{\star r}$. Let $\Sec_{d,r}(X)$ be the $r$th secant scheme of the $d$th Veronese embedding of $X$ with the definition given above. By Proposition~\ref{prop:join-reduced}, if $X$ is geometrically reduced, respectively geometrically integral, then the same is true for $\Sec_{d,r}(X)$. In particular, if $X$ is a variety over an algebraically closed field, then the $r$th secant scheme is also a variety and agrees with the geometric definition in terms of adding points on $X$ which was given in the introduction: given two conical subvarieties $X$ and $Y$, the ideal of the Zariski closure of the set of sums $x + y$ where $x \in X$ and $y \in Y$ is $I_X \star I_Y$ because comultiplication is the map on coordinate rings that corresponds to the addition map $V \times V \to V$.

\begin{theorem} \label{thm:secant-X-main}
Let $X \subseteq V$ be a conical subscheme. There is a function $C_X(r)$, depending on $r$ $($and $X)$, but independent of $d$, such that the ideal of $\Sec_{d,r}(X)$ is generated by polynomials of degree $\le C_X(r)$. 
\end{theorem}

\begin{proof}
Set $B = \Sym(V^*) / I_X$. Then $\cI_B(r)_{d,n}$ is the space of degree $n$ polynomials in the ideal of $\Sec_{d,r}(X)$. By Propositions~\ref{prop:invt-noeth} and ~\ref{prop:secant-biideal}, $\fS(\cI_B(r))$ is generated by finitely many elements $f_1, \dots, f_N$ under $\cdot$ and $*$. Every element of $\fS(\cI_B(r))$ can be written as a linear combination of elements of the form $h \cdot (g * f_i)$ and so for fixed $d$, a set of ideal generators for $\Sec_{d,r}(X)$ can be taken to be the set of all $\fS^{-1}(g * f_i)$ such that $g * f_i \in \cB_{d,n}^\Sigma$ for some $n$. The degree of $g * f_i$ is the same as that of $f_i$ (if $f \in \cB_{d,n}^\Sigma$, then its degree is $n$). So we can take $C_X(r) = \max( \deg(f_1), \dots, \deg(f_N))$.
\end{proof}

\begin{proof}[Proof of Theorems~\ref{thm:secant-main} and \ref{thm:secant-flag}]
The independence from $\dim V$ follows from \cite[Proposition 5.7]{manivel-michalek}. Translating into our notation, the sections of the bundle $\cL(\bd)$ is the Schur functor $\bS_{\lambda(\bd)}(V^*)$ where $\lambda(\bd) = \sum_{i=1}^\ell (d_i^{e_i})$ and $(b^a) = (b, \dots, b)$ (repeated $a$ times). In particular, the number of nonzero parts of $\lambda(\bd)$ is at most $e_\ell$. So from the quoted result, to bound the degrees of the ideal generators, it suffices to consider a vector space of dimension $(r+2) e_\ell$. So if $\ell = 1$ (in particular, for Theorem~\ref{thm:secant-main}), we are done.

Now we explain the case of general $\ell$. We may assume that $V$ is a fixed vector space of dimension $(r+2) e_\ell$. The proof now follows just as in the proof of Theorem~\ref{thm:secant-X}: we replace $B$ by the ring
\[
\bigoplus_{d_1, \dots, d_\ell \ge 0} \rH^0(\bF(\be, V); \cL(\bd)).
\]
This is graded by the sum $d_1 + \cdots + d_\ell$ and is finitely generated in degree $1$ (this is well known, a reference is \cite[Theorem 3.1.2]{brion-kumar}). For every line bundle $\cL(\bd)$, $\cS(B)$ contains the subring $\cS(B(\bd))$ where $B(\bd) = \bigoplus_{d \ge 0} \rH^0(\bF(\be, V); \cL(\bd)^{\otimes d})$, and similarly for $\cS(B)_\Sigma$ and $\cS(B(\bd))_\Sigma$. Furthermore, $\cI_B(1) \cap \cS(B(\bd))_\Sigma = \cI_{B(\bd)}(1)$, and the same holds for the higher secant ideals since $\cS(B(\bd))$ is also closed under the comultiplication.

In particular, $\fS(\cI_B(r))$ is finitely generated by $f_1, \dots, f_N$. An element of $\fS(\cI_{B(\bd)}(r))$ can be written as a linear combination of $h \cdot (g * f_i)$. In particular, for this multiplication to be defined, we need $g * f_i \in \fS(\cS(B(\bd))_\Sigma)$. This implies that a set of ideal generators for the $r$th secant ideal for $\bF(\be,V)$ with respect to the map given by $\cL(\bd)^{\otimes d}$ is the set of $\fS^{-1}(g * f_i)$ such that $g * f_i \in \fS(\cS(B(\bd)_\Sigma))_{d,n}$ where $n = \deg(f_i)$. So as in the proof of Theorem~\ref{thm:secant-X}, we can take $C_\be(r) = \max(\deg(f_1), \dots, \deg(f_N))$.
\end{proof}

\begin{remark}
Alternatively, to get the independence from $\dim V$ in Theorem~\ref{thm:secant-main}, one can use subspace varieties (in fact, they are implicit in the proof of \cite[Proposition 5.7]{manivel-michalek}), which we now explain. If $V$ is a vector space, then the rank $r$ {\bf subspace variety} $\Sub_{d,r}(V)$ of $\rD^d V$ is the collection of vectors that belong to a subspace $\rD^d W$ where $W \subseteq V$ and $\dim W = r$. In fact, $\Sub_{d,r}(V)$ is a closed subvariety of $\rD^d V$ since it can be expressed as the image of a variety under a projective morphism (see \cite[\S 7.1]{weyman}).
Note that $\Sec_{d,r}(V) \subseteq \Sub_{d,r}(V)$: given $\sum_{i=1}^r v_i^{\otimes d} \in \Sec_{d,r}(V)$, it belongs to $\rD^d W$ where $W$ is any $r$-dimensional subspace of $V$ that contains $v_1, \dots, v_r$, and taking Zariski closures is not an issue since $\Sub_{d,r}(V)$ is closed.

It follows from \cite[Proposition 4.2.6]{porras} or \cite[Corollary 7.2.3]{weyman} that the ideal of $\Sub_{d,r}(V)$ has a determinantal description and in fact is generated by polynomials of degree $r+1$. As a representation of $\GL(V)$, the coordinate ring of $\Sub_{d,r}(V)$ is a sum of Schur functors $\bS_\lambda(V^*)$ with $\ell(\lambda) \le r$ \cite[Proposition 7.1.2(b)]{weyman}, and using \cite[Corollary 9.1.3]{expos}, the lattice of $\GL(V')$-equivariant ideals in the coordinate ring of $\Sub_{d,r}(V')$ is isomorphic to the lattice of $\GL(V)$-equivariant ideals in the coordinate ring of $\Sub_{d,r}(V)$ as soon as $\dim V' \ge \dim V \ge r$. So one need only work with $r$-dimensional vector spaces to study $\Sec_{d,r}(V)$.
\end{remark}

\section{Multi-graded generalizations} \label{sec:products}

Fix a positive integer $m$. Given an $m$-tuple of positive integers $\bd = (d_1, \dots, d_m)$ and vector spaces $\bV = (V_1, \dots, V_m)$, the Segre--Veronese cone is the image of the map 
\begin{align*}
V_1 \oplus \cdots \oplus V_m &\to \rD^{d_1} V_1 \otimes \cdots \otimes \rD^{d_m} V_m\\
(v_1, \dots, v_m) &\mapsto v_1^{\otimes d_1} \otimes \cdots \otimes v_m^{\otimes d_m},
\end{align*}
and the $r$th secant variety $\Sec_{\bd,r}(\bV)$ is defined by taking the Zariski closure of the set of $r$-fold sums of vectors in the image.

The techniques used above can be adopted to prove the following generalization:

\begin{theorem} \label{thm:segre-ver}
There is a function $C_m(r)$, depending on $r$, but independent of $\bd$ and $\dim V_i$, such that the ideal of $\Sec_{\bd,r}(\bV)$ is generated by polynomials of degree $\le C_m(r)$.
\end{theorem}

\begin{proof}
We have avoided proving this from the beginning to reduce notational complexity. We just comment on some of the changes.

To get the independence from $\dim V_i$, we can use a multi-graded version of \cite[Proposition 5.7]{manivel-michalek}. Though the result is not stated in the multi-graded case, the proof generalizes without change. The space $\cA$ is replaced by 
\[
\bigoplus_{n, d_1, \dots, d_m} (\Sym^{d_1} \bk^r \otimes \cdots \otimes \Sym^{d_m} \bk^r)^{\otimes n}.
\]
This is $\cS(B)$ where $B = \bk[x_{i,j} \mid 1 \le i \le m,\ 1 \le j \le r]$, so the rest follows from what we've already proven.
\end{proof}

\begin{remark}
We suspect that the functions $C_m(r)$ can be chosen to be independent of $m$, equivalently, $\max_m \{C_m(r)\} < \infty$ for all $r$. For example, it is well-known that $C_m(1) = 2$ for all $m$, and the main result of \cite{raicu} shows that $C_m(2) = 3$ for all $m$. Unfortunately, this seems to be out of reach with our techniques.
\end{remark}

Our techniques also allow us to get a generalization of Theorem~\ref{thm:secant-X}. Given conical subvarieties $X_i \subseteq V_i$, the secant variety $\Sec_{\bd,r}(\bX)$ is the Zariski closure of the set of expressions $\sum_{i=1}^r x_{1,i}^{\otimes d_1} \otimes \cdots \otimes x_{m,i}^{\otimes d_m}$ in $\rD^{d_1}(V_1) \otimes \cdots \otimes \rD^{d_m}(V_m)$, where $x_{j,i} \in X_j$. To get a scheme-theoretic definition, we use the join of ideals as in \S\ref{sec:proof}. Our generalization is as follows:

\begin{theorem} \label{thm:product-schemes}
Let $X_i \subseteq V_i$ be conical subschemes for $i=1,\dots,m$. There is a function $C_{\bX}(r)$, depending on $r$ $($and the $X_i)$, but independent of $\bd$, so that the ideal of $\Sec_{\bd,r}(\bX)$ is generated by polynomials of degree $\le C_\bX(r)$.
\end{theorem}

\begin{proof}
The main step is to use multi-graded versions of the shuffle algebras $\cS(B)$ introduced in \S\ref{ss:shuffle-any} and to deduce finite generation of its ideals from the polynomial case which was discussed in the previous proof. Everything else in \S\ref{sec:coalgebraic} follows in the same way.
\end{proof}

Finally, we have a multi-graded generalization of Theorem~\ref{thm:secant-flag}. Given flag varieties $\bF(\be(1), V_1), \dots, \bF(\be(m), V_m)$, and non-negative integer vectors $\bd(1), \dots, \bd(m)$ of the appropriate length, let $\Sec_{\bd(\bullet); r}(\bF(\be(\bullet), V_\bullet))$ denote the $r$th secant variety of the image of $\bF(\be(1), V_1) \times \cdots \times \bF(\be(m), V_m)$ under the map given by the line bundle $\cL(\bd(1)) \boxtimes \cdots \boxtimes \cL(\bd(m))$.

\begin{theorem} \label{thm:secant-flag-product}
Let $(\be(1), \dots, \be(m))$ be an $m$-tuple of increasing sequences of positive integers. There is a function $C_{\be(\bullet)}(r)$ depending on $r$ $($and $\be(1), \dots, \be(m))$, but independent of $\bd(1), \dots, \bd(m)$ and $\dim V_i$, such that the ideal of $\Sec_{\bd(\bullet); r}(\bF(\be(\bullet), V_\bullet))$ is generated by polynomials of degree $\le C_{\be(\bullet)}(r)$.
\end{theorem}

\begin{proof}
This follows from Theorem~\ref{thm:product-schemes} if we do not require independence from $\dim V_i$. To get the independence, we use  \cite[Proposition 5.7]{manivel-michalek} (as we commented before, it is not stated in the multi-graded case, but the same proof works).
\end{proof}

\end{document}